\documentclass[12pt, a4paper, twosided, reqno]{amsart}
\usepackage{enumerate, cite}
\usepackage{hyperref}

\newtheorem{lemma}{Lemma}

\newtheorem{example}{Example}
\newtheorem{corollary}{Corollary}

\newtheorem{definition}{Definition}
\newtheorem{remark}{Remark}
\newtheorem*{thA}{Theorem A}
\newtheorem*{thB}{Theorem B}
\newtheorem*{thC}{Theorem C}
\newtheorem*{thD}{Theorem D}

\DeclareMathOperator{\RE}{Re}

\allowdisplaybreaks
\author[G. Pant and M. Saini]{Garima Pant and Manisha Saini}
\address{department of mathematics, university of delhi, delhi-110007, india.}
\email{garimapant.m@gmail.com}
\address{department of mathematics, university of delhi, delhi-110007, india.}
\email{msaini@maths.du.ac.in, sainimanisha210@gmail.com}
\thanks {The research work of the second author is supported by research fellowship from University Grants Commission (UGC), New Delhi, India.}

\thanks {The second author is Senior Research Fellow (UGC, New Delhi, India).}
\title[Infinite order solutions of]{Infinite order solutions of second order linear differential equations}

\begin{document}
\maketitle
\begin{abstract}
This article deals with the second order linear differential equations with entire coefficients. We prove some results involving conditions on coefficients so that the order of growth of every non-trivial solution is infinite.
\end{abstract}
\section{Introduction}
A well known result of Herold for  second order linear differential equation 
\begin{equation}\label{sde}
f''+A(z)f'+B(z)f=0
\end{equation}
with entire coefficients $A(z)$ and $B(z)$, where $B(z)\not \equiv 0$ says that, all solutions are entire functions. As we know that the solutions of the differential equation are local in nature so when we say them entire functions it means we are talking in terms of their analytic continuation to the whole plane. The order of growth of these solutions has been a topic of interest for a long time.

H. Wittich has shown that solutions of equation \eqref{sde} have order of growth finite if and only if the coefficients $A(z)$ and $B(z)$ are polynomials. 
It is obvious that if $A(z)$ or $B(z)$ is a transcendental entire function then, the existence of non-trivial solutions with infinite order of growth is guaranteed. It is interesting to note here that if one of the independent solution is of infinite order and other is of finite order then, almost all solutions of equation \eqref{sde} are of infinite order. This leads to the question of finding conditions on coefficients so that all non-trivial solutions of equation \eqref{sde} are of infinite order. To achieve these results the Nevanlinna's value distribution theory is the basic and fundamental tool.  In this article, we are going to use extensively the terms of value distribution theory such as  characteristic function $T(r,f)$, maximum modulus $M(r,f)$, order of growth $\rho(f)$, hyper-order of growth $ \rho_2(f)$ and exponent of convergence $\lambda(f)$ of zeros of a meromorphic function $f$. All this and more can be found in the well known books  ``Meromorphic functions", ``Value distribution theory" and ``Nevanlinna theory and complex differential equations" by  W.K. Hayman, Lo Yang  and Ilpo Laine, respectively.
The standard references for the above paragarphs are  \cite{lainebook}, \cite{wittich}, \cite{haymero} and \cite{yang}.

In our recent work \cite{sm}, we have obtained some conditions on the coefficients of the equation \eqref{sde} so that all non-trivial solutions of the equation have order of growth infinite. These conditions are $\rho(A)\neq \rho(B)$ or $B(z)$ has Fabry gaps, where entire function $A(z)$ possesses zero as a Borel exceptional value and $B(z)$ is transcendental entire function. In our first result, we have simply replaced the zero as a Borel exceptional value by a finite complex number as Borel exceptional value.

\begin{thA}\label{thmbor}
Let $A(z)$ be an entire function with a finite Borel exceptional value and $B(z)$ be a transcendental entire function satisfying any of the following conditions
\begin{enumerate}
\item $\rho(B)\neq \rho(A)$ 
\item $B(z)$ has Fabry gaps.
\end{enumerate}
 Then, every non-trivial solution $f$ of equation \eqref{sde} satisfies
$$\rho(f)=\infty.$$
\end{thA}
Recall that an entire function $f(z)=\sum_{n=0}^{\infty}a_{\lambda_n}z^{\lambda_n}$ has Fabry gaps if $\lim_{n\to \infty}n/\lambda_n= 0.$ One should note that the order of growth of such function is non-zero \cite{hayman}.

We illustrate the theorem given above by some examples.
\begin{example}
The differential equation
 $$f''+(e^{z^2}-1)f'+e^zf=0$$
has all non-trivial solutions of infinite order of growth.
\end{example}
\begin{example}
All non-trivial solutions of the differential equation $f''+e^zf'+B(z)f=0$, where $B(z)$ is an entire function with Fabry gaps, are of infinite order of growth. 
\end{example}
Here, we give an example which shows that that the hypothesis of Theorem [A] are necessary. 
\begin{example}
The differential equation
$$f''+(e^z+1)f'+e^zf=0 $$
possesses a solution $f(z)=e^{-z}$ which is of  finite order of growth.
\end{example}

We know that there is a measure of order of growth associated with a function of infinite order called its hyper-order of growth. The next main result of this article gives the hyper-order of growth of the solutions of equation \eqref{sde}.
\begin{thB}\label{hypthm}
Suppose that the coefficients $A(z)$ and $B(z)$ satisfy the hypothesis of Theorem \rm{[A]}. Then, all non-trivial solutions $f$ of equation \eqref{sde} satisfy
$$\rho_2(f)=\max\{ \rho(A), \rho(B)\}$$
where $\max\{\rho(A),\rho(B)\}$ is a finite quantity.
\end{thB}
Now, we state a consequence of Theorem [A] and Theorem [B].

\begin{corollary}
Suppose that $A(z)$ is an entire function with a finite Picard exceptional value and $B(z)$ is an entire function satisfying the hypothesis of Theorem \rm{[A]}. Then, every non-trivial solution $f$ of equation \eqref{sde} satisfies
$$\rho(f)=\infty.$$ 
Also,
$$ \rho_2(f)=\max\{\rho(A),\rho(B)\}$$
when $\max\{\rho(A),\rho(B)\}$ is a finite quantity.
\end{corollary}
In our third result, we have relaxed the hypothesis of Theorem [A] and added one additional factor to achieve the same conclusion as has been achieved in Theorem [A]. We have also achieved the conclusion of Theorem [B] in terms of hyper-order of growth as well.

\begin{thC}\label{thmborloword}
Suppose that $A(z)$ is an entire function with a finite Borel exceptional value and $B(z)$ is a transcendental entire function such that $\mu(B)\neq \rho(A)$. Then, all non-trivial solutions $f$ of equation \eqref{sde} satisfy 
$$\rho(f)=\infty.$$ 
Also,
$$\rho_2(f)=\max\{\rho(A),\mu(B)\}$$
when $\max\{\rho(A),\rho(B)\}$ is a finite quantity.

\end{thC} 
Here, we would like to mention that the above result is a generalisation of a theorem in our work \cite{sm2}. The following result also holds true.
\begin{corollary}
Let $A(z)$ be an entire function with finite Picard exceptional value and $B(z)$ satisfies the hypothesis of Theorem \rm{[C]}. Then, the conclusion of Theorem \rm{[C]} is also true.
\end{corollary}

The final result of this article is an extension of the work of Kwon and Kim \cite{kwonkim}. They have proved that if $A(z)$ has a finite deficient value and $B(z)$ is an entire function with $\rho(B)<1/2$, then all solutions of equation \eqref{sde} are of infinite order. We have removed the restriction on the order of $B(z)$ given by them.
\begin{thD}\label{defthm}
Suppose that $A(z)$ is transcendental entire function with a finite deficient value and $B(z)$ is a transcendental entire function satisfying any of the following conditions
\begin{enumerate}
\item $B(z)$ has Fabry gaps
\item $T(r,B)\sim\log{M(r,B)}, r\to \infty$, outside a set of finite logarithmic measure.
\end{enumerate}
 Then, 
$$\rho(f)=\infty$$ 
where $f$ is a non-trivial solution of equation \eqref{sde}.
\end{thD}

One can recall that an entire function $f(z)=\sum_{n=0}^{\infty}a_{\lambda_n}z^{\lambda_n}$ has Fej$\acute{e}$r gaps if $$\sum_{n}\lambda_n^{-1}<\infty.$$

In \cite{tmurai}, T. Murai showed that an entire function with Fej$\acute{e}$r gaps satisfies
$$T(r,f)\sim \log{M(r,f)}$$
as $r\to \infty,$ outside a set of finite logarithmic measure.
Therefore, the following corollary is an easy consequence of Theorem [D].
\begin{corollary}
Suppose that the coefficient $A(z)$ satisfies the hypothesis of Theorem \rm{[D]} and $B(z)$ is an entire function with Fej$\acute{e}$r gaps. Then, all non-trivial solutions of equation \eqref{sde} are of infinite order.
\end{corollary}
\section{Preliminary Results}
To make this article self contained we list some results which will help in reading the coming sections.  For the statement of these results, we need the notion of linear measure, logarithmic measure, upper logarithmic density and lower logarithmic density. 
For a set $F\subset(0,\infty)$, we recall these notions as
$$m(F)=\int_F\, dt \qquad \qquad m_l(F) =\int_F \frac{1}{t}\, dt$$
$$\overline{\log dens}(F)=\limsup_{r\to \infty}\frac{m_l(F\cap [1,r))}{\log r}$$
$$\underline{\log dens}(F)=\limsup_{r\to \infty}\frac{m_l(F\cap [1,r))}{\log r},$$
respectively. It is obvious from the definition that $0\leq \underline{\log dens}(F)\leq \overline{\log dens}(F)\leq 1$. 

 The following lemma is an important result and it has been used extensively for proving the results of differential equations in complex domain.
\begin{lemma}\rm{\cite{log gg}}\label{gunlem}
Let $f(z)$ be a transcendental meromorphic function and let $\Gamma= \{ (k_1,j_1), (k_2,j_2), \ldots ,(k_m,j_m) \} $ denote finite set of distinct pairs of integers that satisfy $ k_i > j_ i \geq 0$  for $i=1,2, \ldots,m$. Let $\alpha >1$ and $\epsilon>0$ be given real constants. Then, 
\begin{enumerate}
\item
there exists a set $E_1 \subset[0,2\pi)$ that has linear measure zero, such that if $\psi_0 \in [0,2\pi)\setminus E_1, $ then there is a constant $R_0=R_0(\psi_0)>0$  so that for all $z$ satisfying $\arg z =\psi_0$ and $|z| \geq R_0$ and for all $(k,j)\in \Gamma$, we have 
\begin{equation} \label{ggguneq}
\left| \frac{f^{(k)}(z)}{f^{(j)}(z)}\right|\leq c \left( \frac{T(\alpha r,f)}{r} \log^{\alpha}{r} \log{T(\alpha r,f)} \right)^{(k-j)}.
\end{equation}

If $f(z)$ is of finite order, then $f(z)$ satisfies
\begin{equation} \label{guneq1}
\left|\frac{f^{(k)}(z)}{f^{(j)}(z)}\right| \leq |z|^{(k-j)(\rho(f)-1+\epsilon)}
\end{equation} 
for all $z$ satisfying $\arg z =\psi_0$ and $|z| \geq R_0$ and for all $(k,j)\in \Gamma$.
\item there exists a set $E\subset (1,\infty)$ with $m_l(E)$ is finite  and there exists a constant $c>0$ that depends only on $\alpha$ and $\Gamma$ such that, for all $z$ satisfying $|z|=r\notin E\cup[0,1]$ and for all $(k,j)\in \Gamma$ we have
\begin{equation} \label{ggguneq}
\left| \frac{f^{(k)}(z)}{f^{(j)}(z)}\right|\leq c \left( \frac{T(\alpha r,f)}{r} \log^{\alpha}{r} \log{T(\alpha r,f)} \right)^{(k-j)}.
\end{equation}

If $f(z)$ is of finite order, then $f(z)$ satisfies:
\begin{equation} \label{guneq1}
\left|\frac{f^{(k)}(z)}{f^{(j)}(z)}\right| \leq |z|^{(k-j)(\rho(f)-1+\epsilon)}
\end{equation} 
 for all $z$ satisfying $|z|\notin E\cup [0,1]$ and $|z| \geq R_0$ and for all $(k,j)\in \Gamma$.
\end{enumerate}

\end{lemma}

For the statement of next lemma we need to recall the following.
\begin{definition}\rm{\cite{banklang, sm}}
For a non-constant polynomial $P(z)=a_nz^n+\ldots+a_0$, the term $\delta(P,\theta)$ is defined as: $\delta(P,\theta)=\RE (a_ne^{\iota n \theta})$, where $\RE(z)$ denotes the real part of $z$. A ray $\gamma =re^{\iota\theta} $ is called a critical ray of $e^{P(z)}$ when $\delta(P,\theta)=0$. 
\end{definition}
The rays $\arg{z}=\theta$ such that $\delta(P,\theta)=0$ divides the complex plane into $2n$ sectors of equal length $\pi/n$. Also, $\delta(P,\theta)>0$ and $\delta(P,\theta)<0$ in the alternative sectors. Suppose that $0\leq\phi_1<\theta_1<\phi_2<\theta_2<\ldots<\phi_n<\theta_n<\phi_{n+1}=\phi_1+2\pi$ are $2n$ critical rays of $e^{P(z)}$ satisfying $\delta(P,\theta)>0$ for $\phi_i<\theta<\theta_i$ and $\delta(P,\theta)<0$ for $\theta_i<\theta<\phi_{i+1}$ where $i=1,2,3, \ldots, n$. For $\alpha $ and $\beta$, we fix the notation as follows:
$$\Omega(\alpha, \beta)=\{z:\alpha\leq\arg{z}\leq \beta\}; 0\leq\alpha<\beta\leq2\pi.$$

The following lemma established an estimate for a transcendental entire function of integral order of growth.
\begin{lemma}\label{implem}  \rm{\cite{banklang}}
Let $A(z)=v(z)e^{P(z)}$ be an entire function, where $P(z)$ is a polynomial of degree $n$ and $v(z)$ be an entire function with order less than $n$. Then, for every $\epsilon>0$, there exists $E \subset [0,2\pi)$ of linear measure zero such that

\begin{enumerate}[(i)]

\item for $ \theta \in [0,2\pi)\setminus E $ such that $\delta(P,\theta)>0$, there exists $ R>1 $ such that
\begin{equation}\label{eqA1}
 \exp \{ (1-\epsilon) \delta(P,\theta)r^n \} \leq|A(re^{\iota \theta})| \leq  \exp \{ (1+\epsilon) \delta(P,\theta)r^n \}
\end{equation}
for $r>R.$

\item for $\theta \in [0,2\pi)\setminus E$ such that $\delta(P,\theta)<0$, there exists $R>1$ such that 
\begin{equation}\label{eq2le}
 \exp \{(1+\epsilon) \delta(P,\theta)r^n \} \leq |A(re^{\iota \theta})| \leq \exp \{ (1-\epsilon)\delta(P,\theta) r^n \} 
\end{equation}
for $r>R.$
\end{enumerate}
\end{lemma}
Next lemma is from \cite{besi} and give estimate for an entire function of order less than one.
\begin{lemma}\label{gglemma}
Let $g(z)$ be an entire function of order $\rho$, where $0<\rho<\frac{1}{2}$ and let $\epsilon>0$ be a given constant. Then, there exists a set $F \subset [0,\infty)$ that has upper logarithmic density at least $1-2\rho$ such that $|g(z)| >\exp (|z|^{\rho-\epsilon})$ for all $z$ satisfying $|z| \in F.$ 
\end{lemma}
The following result is an easy consequence of Phragm$\acute{\bf{e}}$n-Lindel$\ddot{\bf{o}}$f theorem.
\begin{lemma}\label{pharlem}\rm{\cite{wang}}
Suppose that $B(z)$ is an entire function with $\rho(B)\in \big[1/2,\infty)$. Then there exists a sector $\Omega(\alpha, \beta),$ $\beta-\alpha\geq \pi/\rho(B)$, such that
$$\limsup_{r\rightarrow \infty}\frac{\log \log |B(re^{\iota \theta})|}{\log r}\geq \rho(B)$$
for all $\theta \in \Omega(\alpha,\beta)$, where $0\leq\alpha<\beta\leq2 \pi.$
\end{lemma}
Next lemma give property of an entire function with Fabry gaps and can be found in \cite{jlongfab}, \cite{zhe}. 
\begin{lemma}\label{fablemma}
Let $g(z)=\sum_{n=0}^{\infty} a_{\lambda_n}z^{\lambda_n}$ be an entire function of finite order with Fabry gaps, and $h(z)$ be an entire function with $\rho(h)=\sigma \in (0,\infty)$. Then for any given $\epsilon\in (0,\sigma)$, there exists a set $F\subset (1,+\infty)$ satisfying $ \overline{\log dens}(F) \geq \xi $, where $\xi\in (0,1)$ is a constant  such that for all $|z| =r \in F$, one has
$$ \log M(r,h) > r^{\sigma-\epsilon}, \quad \log m(r,g) > (1-\xi)\log M(r,g),$$

where $M(r,h)=\max \{ |h(z)|: |z|=r\} $, $m(r,g)=\min \{ |g(z)|: |z|=r\}$ and $M(r,g)= \max \{ |g(z)|: |z|=r\} $.
\end{lemma}
The following remark follows from the above lemma.
\begin{remark} \label{fabremark}
Suppsoe that $g(z)=\sum_{n=0}^{\infty} a_{\lambda_n}z^{\lambda_n}$ be an entire function of order $\sigma \in (0,\infty)$ with Fabry gaps then for any given $\epsilon >0, \quad (0<2\epsilon <\sigma)$, there exists a set $F\subset (1,+\infty)$ satisfying $\overline{\log dens}(F) \geq \xi$, where $\xi \in (0,1)$ is a constant such that for all $|z| =r \in F$ , one has
$$ |g(z)|> M(r,g)^{(1-\xi)}> \exp{\left((1-\xi) r^{\sigma-\epsilon}\right)}>\exp{\left(r^{\sigma-2\epsilon}\right)}.$$
\end{remark}
The next lemma gave the estimate for a meromorphic function.
\begin{lemma}\label{rholem}\rm{\cite{kwonkim}}
Suppose that $g$ is a meromorphic function of order $\rho\in[0,\infty)$.  For given $\zeta>0$ and $0<l<1/2$, there exists a constant $K(\rho,\zeta)$ and a set $F_{\zeta}\subset [0,\infty)$ of lower logarithmic density greater than $1-\zeta$ such that
\begin{equation*}
r\int_{J} \bigg|\frac{g'(re^{\iota \theta})}{g(re^{\iota \theta})}\, d\theta\bigg|<K(\rho,\zeta)\bigg(l\log{\frac{1}{l}}\bigg)T(r,g)
\end{equation*}
for all $r\in F_{\zeta}$ and for every interval $J\subset [0,2\pi)$ of length $l$.
\end{lemma}

The next lemma provided an upper bound for the hyper-order of growth of solutions $f$ of equation \eqref{sde}.
\begin{lemma}\label{wuthm}\rm{ \cite{pz}}
Suppose that $A(z)$ and $B(z)$ are entire functions of finite order. Then, 
\begin{equation*}
\rho_2(f)\leq\max\{ \rho(A), \rho(B)\}
\end{equation*}
for all solutions $f$ of  equation \eqref{sde}.
\end{lemma}
The following lemma involves the central index of a transcendental entire function.
\begin{lemma}\rm{\cite{lainebook}}\label{centlem}
Let $f$ be a transcendental entire function, $0<\delta<1/4$ and $z$ be such that $|z|=r$ and $|f(z)|>M(r,f)\nu(r,f)^{-\frac{1}{4}+\delta}$ holds. Then there exists a set $F\subset (0,\infty)$ with $m_l(F)<\infty$ such that
\begin{equation}\label{centeq}
f^{(m)}(z)=\left(\frac{\nu(r,f)}{z}\right)^m(1+o(1))f(z)
\end{equation}
for all $m\geq 0$ and all $r\notin F$.
\end{lemma}
\section{Results}This section includes those results which we have framed with proofs for getting foothold in the last section. 
\begin{lemma}\label{meszero}
Suppose that $f$ is an entire function such that $T(r,f)\sim\log{M(r,f)}$ as $r\to \infty,$ outside a set of finite logarithmic measure. For $0<c<1$, let
$$E(r)=\{\theta\in[0,2\pi): \log|f(re^{\iota\theta})|\leq (1-c)\log{M(r,f)}\}$$
then, $m(E(r))\to 0$ as $r\to \infty$, outside a set of finite logarithmic measure.

\end{lemma}
\begin{proof} 
This result is already in arxiv [2001.10729].  Proof can be seen there.
\end{proof}
We present a result which follows easily from Lemma [\ref{centlem}].
\begin{lemma}\label{maxlem}
Suppose $f$ is a transcendental entire funtion then, there exists a set $F\subset (0,\infty)$ with finite logarithmic measure such that  for all $z$ satisfying $|z|=r\notin F$ and $|f(z)|=M(r,f)$ we have
\begin{equation}\label{centreq}
\bigg| \frac{f(z)}{f^{(m)}(z)}\bigg|\leq 2r^m
\end{equation}
for all $m\in \mathbb{N}.$
\end{lemma}
\begin{proof}
Let $z$ be a point on the circle $|z|=r$ such that $|f(z)|=M(r,f)$ then, Lemma [\ref{centlem}] implies that, there exists $F\subset(0,\infty)$ with finite logarithmic measure such that 
\begin{equation*}
f^{(m)}(z)=\left(\frac{\nu(r,f)}{z}\right)^m(1+o(1))f(z), \quad  m\in \mathbb{N}
\end{equation*}
 for all $r\notin F$. We observe that $\nu(r,f)\geq 1$ therefore, equation \eqref{centeq} implies that
\begin{equation*}
\bigg| \frac{f(z)}{f^{(m)}(z)}\bigg|\leq \frac{|z|^m}{\nu(r,f)}(1+o(1))\leq 2r^m
\end{equation*}
for all $m\in \mathbb{N}$ and $r\notin F$.
\end{proof}
The next lemma shows that the conclusion of Lemma [\ref{maxlem}] holds true in a neighbourhood of the argument of $z$.
\begin{lemma}\label{reciprolem}
Let $f$ be a transcendental entire function and $z_r=re^{\iota \theta_r}$ be a point such that $|f(z_r)|=M(r,f)$. Then, for sufficiently large $r\notin F$ where $m_l(F)<\infty$, there exists a constant $\delta_r>0$ such that for all $z$ satisfying $|z|=r$ and $\arg{z}=\theta\in [\theta_r-\delta_r,\theta_r+\delta_r]$, one has
\begin{equation*}
\bigg| \frac{f(z)}{f^{(m)}(z)}\bigg|\leq 2r^m
\end{equation*}
for all $m\in \mathbb{N}.$
\end{lemma}
\begin{proof}
The point $z_r=re^{\iota \theta_r}$ satisfies $|f(z_r)|=M(r,f)$. The function $|f(z)|$ is continuous on $|z|=r$ thus,  there exists $\delta_r$ such that for all $z$ satisfying $|z|=r$ and $\arg{z}=\theta\in[\theta_r-\delta_r,\theta_r+\delta_r]$, we have
$$||f(z)|-|f(z_r)|| <\frac{|f(z_r)|}{2}.$$
This again implies $$|f(z)|>\frac{|f(z_r)|}{2}=\frac{M(r,f)}{2}>M(r,f)\nu(r,f)^{-1/4+\delta}$$
for some $0<\delta<1/4$. Therefore, Lemma [\ref{centlem}] yields that   
 \begin{equation*}
f^{(m)}(z)=\left(\frac{\nu(r,f)}{z}\right)^m(1+o(1))f(z), \quad  m\in \mathbb{N}
\end{equation*}
for all $ |z|=r\notin F$ and $\arg{z}=\theta\in [\theta_r-\delta_r,\theta_r+\delta_r]$. Now, the conclusion  follows from Lemma [\ref{maxlem}].
\end{proof}
\section{Proofs of Theorems}
\begin{proof}[\underline{Proof of Theorem A}] Let $a\in\mathbb{C}$ be the Borel exceptional value of the function $A(z)$. Therefore, the function $A(z)-a$ is an entire function which has zero as its Borel exceptional value. By employing the Weierstrass factorisation theorem, we can write
\begin{equation}\label{eqA}
A(z)-a=v(z)e^{P(z)}
\end{equation}
where $P(z)$ is a non-constant polynomial of degree $n$ and $v(z)(\not\equiv 0)$ is an entire function such that $\rho(v)<n$. Lemma [\ref{implem}] implies that for $\epsilon>0$ there exists $E\subset [0,2\pi)$ with $m(E)=0$ such that for $\theta \in [0,2\pi)\setminus E$ with $\delta(P,\theta)<0$ there exists $R_0>1$ such that
\begin{equation}\label{A-aeq}
|A(z)-a|\leq \exp\{ (1-\epsilon)\delta(P,\theta) r^n \}
\end{equation}
for all $z=re^{\iota \theta}$ and $r>R_0$.
Suppose that $f$ is a finite order solution of the equation \eqref{sde}. Therefore, by Lemma \ref{gunlem}, there exists a set $E_1\subset [0,2\pi)$ with $m(E_1)=0$, such that if $\theta\in [0,2\pi)\setminus E_1$, then there is a constant $R=R(\theta)$ so that for $z$ satisfying $\arg{z}=\theta$ and $|z|>R$, we have
\begin{equation}\label{eqf}
\big|\frac{f^{(k)}(z)}{f(z)}\big|\leq |z|^{2\rho(f)}, k=1,2.
\end{equation}

\begin{enumerate}
\item Suppose that $\rho(A)<\rho(B)$ then, the conclusion holds using the Theorem [2] in \cite{finitegg}. Therefore, we need to consider $\rho(B)<\rho(A)$. The following three cases are there to be discussed:
\begin{enumerate}

\item when $0<\rho(B)<1/2$ then, by the Lemma \ref{gglemma}, for $0<\alpha<\rho(B)$  there exists a set $F\subset(0,\infty) $ with $\overline{\log dens} (F)>1-\rho(B)$ such that 
\begin{equation}\label{eqB}
|B(re^{\iota \theta})|>\exp\{r^{\alpha}\}
\end{equation}
for all $r\in F$. From the equation \eqref{sde},  \eqref{A-aeq},  \eqref{eqf} and  \eqref{eqB}  for all $z$ satisfying $\theta \in [0,2\pi)\setminus \big(E_1\cup E\big)$ with $\delta(P,\theta)<0$ and $|z|=r\in F$, $r>R_0$ we have
\begin{align*}
\exp\{r^{\alpha}\}&< |B(re^{\iota \theta})|\\
&\leq \big|\frac{f''(re^{\iota\theta})}{f(re^{\iota\theta})}\big|+|A(re^{\iota \theta})|\big|\frac{f'(re^{\iota\theta})}{f(re^{\iota\theta})}\big|\\
&\leq r^{2\rho(f)}\big(1+|A(re^{\iota \theta})-a|+|a|\big)\\
&\leq   r^{2\rho(f)}\big(1+o(1)+|a|\big)
\end{align*}
which implies a contradiction for arbitrary large $r$.
\item If $\rho(B)\in [\frac{1}{2},\infty)$ then, using Lemma [\ref{pharlem}], there exists a sector $\Omega(\alpha, \beta); 0\leq\alpha<\beta\leq2\pi$ with $\beta-\alpha \geq\frac{\pi}{\rho(B)}$ such that
\begin{equation}\label{eqqB}
\limsup_{r\to \infty}\frac{\log \log |B(re^{\iota \theta})|}{\log{r}}=\rho(B)
\end{equation}
for all $\theta \in [\alpha,\beta]$.  We have that $\rho(B)<\rho(A)$, therefore there exists $\theta_0 \in [\alpha,\beta]\setminus \big(E\cup E_1\big)$ such that the equation \eqref{A-aeq} holds true for $\theta_0$. We get a  contradiction using equations \eqref{sde}, \eqref{A-aeq}, \eqref{eqf} and \eqref{eqqB}.
\item When $\rho(B)=0$ then, a result of \cite{pd} implies that
\begin{equation}
\limsup_{r\to \infty}\frac{\log |B(re^{\iota \theta})|}{\log{r}}=\infty
\end{equation}
for all $\theta\in[0,2\pi)$. This will implies that 
\begin{equation}\label{eqqqB}
r^M<|B(re^{\iota \theta})|
\end{equation}
for all $r>R(M)$ and $\theta\in [0,2\pi)$, where $M$ is an arbitrary large constant. Therefore, for all sufficiently large $r$ and $\theta_0\in [0,2\pi)\setminus \big(E\cup E_1\big)$ with $\delta(P,\theta_0)<0$, using equations \eqref{sde}, \eqref{A-aeq}, \eqref{eqf} and \eqref{eqqqB}, we get a contradiction.
\end{enumerate}
\item Suppose that $B(z)$ is a function with Fabry gaps. Then, Remark [\ref{fabremark}] for $0<\alpha<\rho(B)$, there exists a set $F\subset (0,\infty)$ with $\overline{\log dens}(F)>0$ such that
\begin{equation}\label{eqqqqB}
|B(re^{\iota \theta}|>\exp\{r^{\alpha}\}
\end{equation}
for all $r\in F$. Therefore, for sufficiently large $r\in F$ and $\theta_0\in [0,2\pi)\setminus\big(E\cup E_1)$ with $\delta(P,\theta_0)<0$, by the equations \eqref{sde},  \eqref{A-aeq},  \eqref{eqf} and  \eqref{eqqqqB}, we get a contradiction.
\end{enumerate}
\end{proof}
\begin{proof}[\underline{Proof of Theorem B}]

\begin{enumerate}
\item We know that all solutions $f(\not \equiv 0)$ of equation \eqref{sde} are of infinite order, when $\rho(B)\neq \rho(A)$ by Theorem [A]. Using Lemma [\ref{gunlem}], for $\epsilon>0$, there exists a set $E\subset [1,\infty)$ that has finite logarithmic measure such that for all $z$ satisfying $|z|=r\notin E\cup [0,1]$ we have
\begin{equation}\label{eqfi}
\left| \frac{f^{(k)}(z)}{f^{(j)}(z)}\right| \leq c  \left[T(2r,f)\right]^{2(k-j)}
\end{equation}
where $c>0$ is a constant.

If $\rho(A)<\rho(B)$ then from \cite[Theorem 1]{kwon} and Lemma [\ref{wuthm}] we get that $\rho_2(f)= \max\{ \rho(A),\rho(B) \} $.

 If $\rho(B)<\rho(A)=n, n\in \mathbb{N},$  one can choose $\beta$ such that $\rho(B)<\beta<\rho(A)$. Suppose that $a\in \mathbb{C}$ is a Borel exceptional value of $A(z)$. Therefore, $A(z)-a=v(z)e^{P(z)}$, where $P(z)$ is a non-constant polynomial of degree $n$ and $v(z) $ is an entire function with $\rho(v)<n$. 

From Lemma [\ref{maxlem}], we can choose a sequence $r_m\notin E\cup[0,1]\cup F$ such that $r_m\to_{m\to \infty} \infty$ and let $|f(r_me^{\iota \theta_m})|=M(r_m,f)$. Using Lemma [\ref{maxlem}], we obtain
\begin{equation}\label{centraeq}
\bigg| \frac{f(r_me^{\iota \theta_m})}{f'(r_me^{\iota \theta_m})}\bigg|\leq 2r_m.
\end{equation}

Suppose that there exists a subsequence $(\theta_m)$ such that \\
$\lim_{m\to \infty}\theta_m=\theta_0.$ We have the following cases to discuss.
\begin{enumerate}[(i)]
\item If $\delta(P,\theta_0)>0$ then, the continuity of $\delta(P,\theta)$ implies that
$$\frac{1}{2}\delta(P,\theta_0)<\delta(P,\theta_m)<\frac{3}{2}\delta(P,\theta_0)$$
for all sufficiently large $m\in \mathbb{N}$. The part (i) of Lemma [\ref{implem}]  gives
\begin{align}\label{Aeq}
\exp{\left(\frac{1}{2}(1-\epsilon)\delta(P,\theta_0)r^n_m\right)}&\leq |A(r_me^{\iota \theta_m})-a|\\ \notag
&\leq \exp{\left(\frac{3}{2}(1+\epsilon)\delta(P,\theta_0)r^n_m\right)}
\end{align}
for all sufficiently large $m\in \mathbb{N}$. The equations \eqref{sde},  \eqref{centraeq}, \eqref{Aeq} and \eqref{eqfi}, for $z_m=r_me^{\iota \theta_m}$ implies that

\begin{align*}
\qquad \exp{\left(\frac{1}{2}(1-\epsilon)\delta(P,\theta_0)r^n_m\right)}&\leq |A(r_me^{\iota \theta_m})-a|\\\
&\leq \left|\frac{f''(r_me^{\iota \theta_m})}{f'(r_me^{\iota \theta_m)}}\right|\\
&+|B(r_me^{\iota \theta_m})|\left| \frac{f(r_me^{\iota \theta_m})}{f'(r_me^{\iota \theta_m})} \right|+|a| \\
&\leq c\left[T(2r_m,f)\right]^2+\exp{\left(r_m^{\beta}\right)}2r_m+|a|.
\end{align*} 
As $\beta < n,$ this implies that 
\begin{equation}\label{eqthm1}
\limsup_{m\rightarrow \infty}\frac{\log^+ \log^+ T(r_m,f)}{\log{r_m}} \geq \rho(A).
\end{equation}
Using Lemma [\ref{wuthm}] and equation \eqref{eqthm1} we have
$$ \rho_2(f)=\max\{ \rho(A),\rho(B)\}. $$
\item If $\delta(P,\theta_0)<0$ then,  the continuity of $\delta(P,\theta)$ implies that
$$ \frac{3}{2}\delta(P,\theta_0)<\delta(P,\theta_m)<\frac{1}{2}\delta(P,\theta_0)$$
for sufficiently large $m\in \mathbb{N}$.  Using part (ii) of Lemma [\ref{implem}] we obtain
\begin{align}\label{Aeq1}
\exp{\left(\frac{3}{2}(1+\epsilon)\delta(P,\theta_0)r^n_m\right)}&\leq |A(r_me^{\iota \theta_m})-a|\\ \notag
&\leq \exp{\left(\frac{1}{2}(1-\epsilon)\delta(P,\theta_0)r^n_m\right)}
\end{align}
for all sufficiently large $m\in \mathbb{N}$. The equations \eqref{sde}, \eqref{eqfi} \eqref{centraeq}  and \eqref{Aeq1}, for $z_m=r_me^{\iota \theta_m}$ implies that
\begin{align*}
\qquad\exp{\left(\frac{3}{2}(1+\epsilon)\delta(P,\theta_0)r^n_m\right)}&\leq |A(z_m)-a|\\
&\leq \left|\frac{f''(r_me^{\iota \theta_m})}{f'(r_me^{\iota \theta_m)}}\right|\\
&+|B(r_me^{\iota \theta_m})|\left| \frac{f(r_me^{\iota \theta_m})}{f'(r_me^{\iota \theta_m})} \right| +|a|\\
&\leq c\left[T(2r_m,f)\right]^2+\exp{\left(r_m^{\beta}\right)}2r_m+|a| .
\end{align*}
Therefore, as earlier we have $\rho_2(f)=\max\{\rho(A),\rho(B)\}$.
\item If $\delta(P,\theta_0)=0$. For all $m\in \mathbb{N}$, we use Lemma [\ref{reciprolem}] to have an interval $(\theta_m-l_0,\theta_m+l_0), 0<l_0<1/2$ such that, for all $\theta\in (\theta_m-l_0,\theta_m+l_0)$ we have  
\begin{equation}\label{receq}
\frac{|f(r_me^{\iota \theta})|}{|f'(r_me^{\iota\theta})|} \leq 2r_m.
\end{equation}
Again choose $\theta^*_m$ such that $\theta^*_m\in (\theta_m+l_0/3,\theta_m+l_0)$ and $\theta^*_m\to \theta^*_0$ as $m\to \infty$. This implies $\theta^*_0\in(\theta_0+l_0/3,\theta_0+l_0)$. 

Let us consider $\delta(P,\theta^*_0)>0$ then, again by the continuity of $\delta(P,\theta)$ we have
\begin{align}\label{del0}
\exp{\left(\frac{1}{2}(1-\epsilon)\delta(P,\theta_0^*)r^n_m\right)}&\leq |A(r_me^{\iota \theta^*_m})-a|\\
&\leq \exp{\left(\frac{3}{2}(1+\epsilon)\delta(P,\theta_0^*)r^n_m\right)}\notag
\end{align}
for sufficiently large $m$. Applying equations \eqref{sde}, \eqref{eqfi}, \eqref{receq} and \eqref{del0} , for $z^*_m=r_me^{\iota \theta^*_m}$, we get the desired result.

If $\delta(P,\theta^*_0)<0$ then, as in the case (ii) we can show that the conclusion holds true .
\end{enumerate}

\item It has been already proved that all non-trivial solutions $f(z)$ of  equation \eqref{sde}, with $A(z)$ and $B(z)$ satisfying the hypothesis of Theorem [A], are of infinite order. Also, if $\rho(A)\neq \rho(B)$ then, using case (1) above,  
$$\rho_2(f)=\max\{ \rho(A), \rho(B)\}.$$
Now, let $\rho(A)=\rho(B)=n, n\in\mathbb{N}$. Using  Remark [\ref{fabremark}], for $\epsilon>0$, there exist $F\subset (1,\infty)$ satisfying $\overline{\log dens}(F)> 0$ such that for all $|z|=r\in F$ we have 
\begin{equation}\label{eqfab}
|B(z)|>\exp{\{r^{n-\epsilon}\}}.
\end{equation}
Now, we choose $\theta \in [0,2\pi)$ with $\delta(P,\theta)<0$ and a sequence $(r_m) \subset F \setminus \left(E\cup[0,1]\right)$ satisfying $r_m\to\infty.$ Using equations \eqref{sde}, \eqref{eq2le},  \eqref{eqfi} and \eqref{eqfab} we obtain
\begin{align*}
\qquad \quad \exp{\left(r_m^{n-\epsilon}\right)}&< |B(r_me^{\iota \theta})|\leq \left| \frac{f''(r_me^{\iota \theta})}{f(r_me^{\iota \theta})}\right| +|A(r_me^{\iota \theta})| \left| \frac{f'(r_me^{\iota \theta})}{f(r_me^{\iota \theta})}\right| \\
& \leq c T(2r_m,f)^4 \\
&+ (\exp{( (1-\epsilon)\delta(P,\theta)r_m^{n}) }+|a|)cT(2r_m,f)^2 \\
&\leq cT(2r_m,f)^2(1+o(1)).
\end{align*}
Thus, we conclude that 
\begin{equation}\label{hypor}
\limsup_{m\rightarrow \infty} \frac{\log^+ \log^+ T(r_m,f)}{\log r_m} \geq n.
\end{equation}
Using Lemma [\ref{wuthm}] and equation \eqref{hypor} we get,
$$\rho_2(f)=\max\{ \rho(A), \rho(B) \}.$$ 
\end{enumerate}
\end{proof}

\begin{proof}[\underline{Proof of Theorem C}]
Let $a\in \mathbb{C}$ is a finite Borel exceptional value then, $A(z)-a=v(z)e^{P(z)}$ where $P(z)$ is a non-constant polynomial and $v(z)$ is an entire function with $\rho(v)<\deg{P}.$ Now, we can apply the arguments of Theorem [1] in \cite{sm2} to get the conclusion of the theorem.

\end{proof}
\begin{proof}[\underline{Proof of Theorem D}]
 Suppose $a\in \mathbb{C}$ is the finite deficient value of the function $A(z)$. Thus, $$\liminf_{r\to \infty}\frac{m\left(r,\frac{1}{A-a}\right)}{T(r,A)}=2d>0.$$
This yields that, for all sufficiently large $r$, we have
$$m\left(r,\frac{1}{A-a}\right)>dT(r,A).$$
Thus, for all sufficiently large $r$, there exists a $\theta_r$ such that for $z_r=re^{\iota \theta_r}$, one may have
\begin{equation}
\log{|A(z_r)-a|}\leq -dT(r,A).
\end{equation}
In Lemma [\ref{rholem}], let $\zeta>0$ and choose $0<l<1/2$  so that $K(\rho(A),\zeta)\big(l\log{\frac{1}{l}}\big)$ is sufficiently small and choose $\phi>0$ with $|\theta_r-\phi|\leq l$ such that 
\begin{align*}
\log{|A(re^{\iota \theta})-a|}&=\log{|A(re^{\iota \theta_r})-a|}+\int_{\theta_r}^{\theta}\frac{d}{dt}\log{|A(re^{\iota t})-a|}\\
& \leq-dT(r,A)+r\int_{\theta_r}^{\theta}\bigg|\frac{(A-a)'(re^{\iota t})}{(A-a)(re^{\iota t})}\bigg|\, dt\\
&\leq-dT(r,A)+K(\rho(A),\zeta)\bigg(l\log{\frac{1}{l}}\bigg)T(r,A)\\
&\leq 0
\end{align*}
for all sufficiently large $r\in F_{\zeta}$ where $\underline{\log dens}(F_{\zeta})\geq 1-\zeta$ and for all $\theta \in [\theta_r-\phi,\theta_r+\phi]$. Hence, for all  sufficiently large $r\in F_{\zeta}$ and $\theta \in [\theta_r-\phi,\theta_r+\phi]$ we have
\begin{equation}\label{Aaeq}
|A(re^{\iota \theta})|\leq |a|+1.
\end{equation}
\begin{enumerate}
\item Let $B(z)$ be  a transcendental entire function with Fabry gaps. Then, Remark [\ref{fabremark}] implies that  for $0<\alpha<\rho(B)$, there exists $F\subset(0,\infty) $ such that $\overline{\log dens}(F)>0$ and equation \eqref{eqqqqB} holds true for all $r\in F$. We choose $\zeta>0$ sufficiently small such that $\overline{\log dens}(F)>\zeta$. We know that 
$$\chi_{F\cap F_{\zeta}}=\chi_F+\chi_{F_{\zeta}}-\chi_{F\cup F_{\zeta}}$$
and $\overline{\log dens}(F\cup F_{\zeta})\leq1$, therefore
\begin{align*}
\overline{\log dens}(F\cap F_{\zeta})&\geq \overline{\log dens}(F)+\underline{\log dens} (F_{\zeta})-\overline{\log dens}(F\cup F_{\zeta})\\
&\geq \overline{\log dens}(F)+1-\zeta -1\\
&>0.
\end{align*}
Hence, for all $r\in F\cap F_{\zeta}$ and $\theta \in [\theta_r-\phi,\theta_r+\phi]\setminus E_1$,
applying equations \eqref{sde}, \eqref{eqf}, \eqref{eqqqqB} and \eqref{Aaeq} we get
\begin{align*}
\exp\{r^{\alpha}\}&<|B(re^{\iota\theta})|\\
&\leq \bigg|\frac{f''(re^{\iota\theta})}{f(re^{\iota\theta})}\bigg|+|A(re^{\iota \theta})|\bigg|\frac{f'(re^{\iota\theta})}{f(re^{\iota\theta})}\bigg|\\
&\leq r^{2\rho(f)}\big(1+|a|+1\big).
\end{align*}
This is a contradiction for sufficiently large $r\in F\cap F_{\zeta}.$
\item Suppose that $B(z)$ is a transcendental entire function satisfying 
$$T(r,B)\sim \log M(r,B))$$
outside a set $H$ of finite logarithmic measure. Using Lemma [\ref{meszero}], for a given $0<c<1$, we get
\begin{equation}\label{Beq}
\log{|B(re^{\iota \theta})|}\geq (1-c)\log{M(r,B)}
\end{equation}
for sufficiently large $r\notin H$ and $\theta \notin J(r)$.  Applying equations \eqref{sde}, \eqref{eqf}, \eqref{Aaeq} and \eqref{Beq}, for all  sufficiently large $r\in F_{\zeta}\setminus H$ and $\theta \in [\theta_r-\phi,\theta_r+\phi]\setminus(E_1\cup J(r))$, we get 
\begin{align*}
M(r,B)^{1-c}&\leq |B(re^{\iota \theta})|\\
&\leq \bigg|\frac{f''(re^{\iota\theta})}{f(re^{\iota\theta})}\bigg|+|A(re^{\iota \theta})|\bigg|\frac{f'(re^{\iota\theta})}{f(re^{\iota\theta})}\bigg|\\
&\leq r^{2\rho(f)}\big(1+|a|+1\big).
\end{align*}
This gives that $M(r,B)\leq r^{4\rho}(2+|a|)^2$ for sufficiently large $r\in F_{\zeta}\setminus H$. This will be a contradiction, as $B(z)$ is a transcendental entire function.
\end{enumerate}
\end{proof}
{\bf{Acknowledgement:}} We are thankful to our thesis advisor Sanjay Kumar for valuable suggestions.

\end{document}